\documentclass[12pt, reqno]{amsart}
\usepackage{amsmath, amsthm, amscd, amsfonts, amssymb, graphicx, color}
\usepackage[bookmarksnumbered, colorlinks, plainpages]{hyperref}
\hypersetup{colorlinks=true,linkcolor=red, anchorcolor=green, citecolor=cyan, urlcolor=red, filecolor=magenta, pdftoolbar=true}

\theoremstyle{definition}

\theoremstyle{remark}

\numberwithin{equation}{section}

\newtheorem{lem}{\textbf{Lemma}}[section]
\newtheorem{thm}[lem]{\textbf{Theorem}}

\theoremstyle{definition}

\theoremstyle{definition}
\theoremstyle{remark}

\renewcommand{\sc}{\mathcal{S}}

\begin{document}

\setcounter{page}{1}

\title[ Some  properties of $T_{\mu}$]{Some  properties of the mapping  $T_{\mu}$    introduced by  a representation  in Banach and locally
convex spaces}

\author[ Soori ]{Ebrahim  Soori$^{1, *}$}

\address{   $^{1}$Department  of Mathematics, Lorestan University, Khoramabad,  Lorestan, Iran.
 }
\email{\textcolor[rgb]{0.00,0.00,0.84}{sori.e@lu.ac.ir; sori.ebrahim@yahoo.com(E. Soori)}}


 \thanks{   $^{*}$Corresponding author}



\begin{abstract}
 Let $ \sc=\{T_{s}:s\in S\} $  be   a representation of  a semigroup $S$.    First,     we prove that the mapping  $T_{\mu}$ introduced by a mean on a subspace  of $l^{\infty}(S)$    has  many properties of  the mappings in  the representation $ \sc$, in Banach  spaces. Then we consider a directed graph and then we define a $Q$-$G$-nonexpansive mapping in locally convex spaces  and show that   $T_{\mu}$ is a $Q$-$G$-nonexpansive mapping if $T_{s}$ is a $Q$-$G$-nonexpansive mapping for each $s\in S$. Then we define $Q$-$G$-attractive point of $ \sc$ and show if a point  $a$ is a  $Q$-$G$-attractive point of $ \sc$ then $a$ is  a  $Q$-$G$-attractive point of $T_{\mu}$.
\end{abstract} \maketitle

\section{Introduction}
Consider   a  reflexive Banach space $E$,   a nonempty closed  convex subset  $ C $ of $E$,     a semigroup $ S $ and   a   representation  $ \sc=\{T_{s}:s\in S\} $  of $S$ as    self  mappings on $C$  and  let $X$ be a subspace of $l^{\infty}(S)$   and $\mu$ be a mean on $X$.
We write $T_{\mu}x $ instead of
$\int\!T_{t}x\, d\mu(t)\/$.
 The relations between the  representation $ \sc$ and  the mapping  $T_{\mu}$ have been interesting for many years. For example we can see \cite{lz, lztt, saeidi, eees}.

In this paper, we study some relations between  the representation $ \sc$ and $T_{\mu}$ in Banach   and locally convex spaces.
\section{Methods  and preliminaries}
  The   space of all bounded real-valued functions defined on $ S $ with supremum
norm is denoted by  $l^{\infty}(S)$.   $l_{s}$ and  $r_{s}$ in $l^{\infty}(S)$  are defined  as follows:  $  (l_{t}g )(s) = g (ts)$ and   $(r_{t}g )(s) = g (st)$,  for all $ s \in S$,  $ t \in S$  and $g\in l^{\infty}(S)$.\\
Suppose that $ X$ is a subspace of $l^{\infty}(S)$ containing 1 and let $ X^{*}$ be its topological  dual space. An element $m $ of $X^{*} $ is said to be a mean on $X$, provided  $\|m\|=m(1)=1$.  For $ m\in X^{*}$ and $g \in X$,   $m_{t}(g(t))$ is often  written instead of $ m(g )$. Suppose that $ X $ is left invariant (respectively, right
invariant), i.e., $ l_{t}(X) \subset X$ (respectively, $ r_{t}(X) \subset X$) for each $ s \in S$. A mean $m$ on $X$ is called    left invariant (respectively, right invariant), provided $m(l_{t}g ) =m(g )$ (respectively, $m(r_{t}g)= m(g )$) for each $t \in S $ and $g \in X$. $X $ is    called      left (respectively, right) amenable if $X$ possesses a left (respectively, right) invariant mean. $X$ is amenable, provided $X$ is both left and right amenable. \\ Let $D$ be a directed set in $X$. A net $\{m_{\alpha}:\alpha\in D\}$ of means on $X$ is called  left regular, provided
\[
\mathop{lim}\limits_{\alpha\in D} \|l^{*}_{t}m_{\alpha}-m_{\alpha}\|= 0,
\]
for every $t\in S$, where $l_{t}^{*}$ is the adjoint operator of $l_{t}$.

Let   $E$ a reflexive Banach space. Let $g$ be a  function on  $S$  into $E$ such that the weak closure of $\{g(s):s\in S\}$ is weakly
compact and suppose that  $X$ is a subspace of $l^{\infty}(S)$ owning  all the functions $s \rightarrow \left<g(s),x^{\ast}\right>$ with  $ x^{\ast} \in E^{\ast}  $. We know from \cite{hi} that, for any $ m\in X^{\ast}$, there
exists a unique element $ g_{m}$ in $E$ such that $\left<g_{m},x^{\ast}\right>=m_{s}\left<f(s),x^{\ast}\right>$
for all $ x^{\ast}\in E^{\ast} $. We denote such $g_{m}$ by $\int g(s) m(s)$. Moreover, if $ m$ is a mean on $X$, then from \cite{ki}, $\int g(s) m(s) \in \overline{\rm{co}}\,\{g(s):s\in S\})$, where $\overline{\rm{co}}\,\{g(s):s\in S\})$ denotes the closure of the convex hull of $\{g(s):s\in S\}$.

Recall the following definitions:
\begin{enumerate}
\item
suppose that $S$ be semigroup.  Let $C$ be a nonempty closed and convex subset of $E$. Then, a family   $ \sc=\{T_{s}:s\in S\} $ of mappings from $C$ into itself is called  a  representation of $S$ as   mappings on $C$ into itself provided    $T_{st}x=T_{s}T_{t}x$ for all $s, t \in S$ and $x\in C$.
Note that,    Fix($ \sc$) is
the set of common fixed points of $\sc$, that is  \\ Fix($ \sc$)=$\bigcap_{ s\in S}\{x\in C: T_{s}x=x  \}$,

\item  Let $E$ be a real Banach space and $C$ be a subset of $E$. We denote by $\text{Fix}(T)$ the
set of fixed points of a mapping $T : C \rightarrow C$. In this paper, a mapping $T : C\rightarrow C$
is called:
\begin{enumerate}
\item nonexpansive provided  $\|Tx - Ty\| \leq \|x - y\|$ for all $x, y \in C$;
\item  quasi
nonexpansive  \cite{saeidi} provided $\|Tx - f\| \leq \|x - f\|$ for all $x \in C$ and $f \in \text{Fix}(T)$;
\item strongly
quasi nonexpansive \cite{saeidi} provided $\|Tx - f\| \leq \|x - f\|$ for all $x \in  C \setminus \text{Fix}(T)$ and $f \in \text{Fix}(T)$;
\item  $F$-quasi nonexpansive (for a subset $F \subseteq \text{Fix}(T)$) provided \\ $\|Tx- f\| \leq \|x - f\|$ for all
$x \in C$ and $f \in  \text{Fix}(T)$;
\item strongly $F$-quasi nonexpansive \cite{saeidi} (for a subset $F \subseteq \text{Fix}(T)$)
provided  $\|Tx - f\| \leq \|x - f\|$ for all $x \in  C \setminus \text{Fix}(T)$ and $f   \in \text{Fix}(T)$, and
\item retraction \cite{saeidi} provided $T^{2} = T$,
\end{enumerate}

\item
 let  $E$ be a Banach space and $C \subset E$. A mapping $T: C\rightarrow C$ is called asymptotically nonexpansive \cite{lz} provided  for all $x, y\in C$ the following inequality holds:
    \begin{equation}\label{asymn}
      \limsup _{n\rightarrow \infty} \|T^{n}x-T^{n}y\| \leq \|x-y\|,
    \end{equation}
\item Suppose that $ \sc=\{T_{s}:s\in S\} $  is a representation of a semigroup $S$ on a set $C$ in a Banach space $E$. An element $a \in E$ is called  an asymptotically attractive point of $S$ for $C$ provided
  \begin{equation}\label{asymnep attr}
      \limsup _{n\rightarrow \infty} \|a-T^{n}_{t}x\| \leq \|a-x\|,
    \end{equation}
    for all $t \in S$ and $x \in C$,

\item consider     a   family of   seminorms  $Q$  on a        locally
convex space $X$ which determines the topology of $X$ and       a nonempty closed and convex subset $C$ of    $X$. Let
$G = (V(G),E(G))$ be a directed graph such that $V(G) =C$ ( to see more details refer to \cite{aka}).  A mapping $ T$ of $ C $ into itself is called  $Q$-$G$-nonexpansive
 if $q(Tx - Ty) \leq q(x - y)$,  whenever $(x, y) \in  E(G)$ for any   $x, y \in C$ and $q \in Q$, and a mapping $f$ is a $Q$-contraction on $E $ if  $ q(f (x) -f (y)) \leq \beta q(x - y)$,  for all $x, y \in E$  such that  $0 \leq \beta < 1$,

 \item consider     a   family of   seminorms  $Q$  on a        locally
convex space $X$ which determines the topology of $X$.
   The locally convex topology $\tau_{Q}$ is separated if and only if the family of seminorms
$Q$ possesses the following property:
for each  $x \in  X \backslash \{0\}$ there exists $q \in Q$ such that $q(x)  \neq 0$ or equivalently   $\displaystyle\bigcap_{q \in Q} \{ x \in X: q(x)=0\}= \{0\}$ ( see \cite{barbu}),
\item  suppose that $ \sc=\{T_{s}:s\in S\} $  is a representation of a semigroup $S$ on a subset $C$ of  a Banach space $E$.  $ \sc$ is   called a uniformly asymptotically nonexpansive  representation, if  for each  $x,y \in C$,
    \begin{align}\label{dhhghj}
     \limsup_{n\rightarrow \infty}  \sup_{t}
\|T_{t}^{n}x - T_{t}^{n}y\|  \leq    \|x-y\|.
    \end{align}
\item  suppose that $ \sc=\{T_{s}:s\in S\} $  is a representation of a semigroup $S$ on a subset $C$ of  a Banach space $E$ and $a$ be an element in $C$.  $ a$ is   called a uniformly asymptotically attractive   point, if  for each  $x  \in C$,
    \begin{align}\label{dhhghj2}
     \limsup_{n\rightarrow \infty}  \sup_{t}
\|a - T_{t}^{n}x\|  \leq    \|a-x\|.
    \end{align}
 \end{enumerate}
The following Lemma which we will use, is well known.
 \begin{lem}\cite{qwq, hi}\label{chimig} Suppose that $g$ is a function of $S$ into $E$ such that the weak closure of $\{ g (t) : t \in S\}$ is weakly compact and let $X$ be a subspace of $B(S)$ containing all the functions $t \rightarrow \langle g (t), x^{*}\rangle$ with $x^{*} \in E^{*}$. Then, for any $\mu \in X^{*}$, there
exists a unique element $g_{\mu}$ in $E$ such that
$$\langle g_{\mu}, x^{*}  \rangle= \mu_{t} \langle g(t) , x^{*}  \rangle$$
for all $x^{*} \in E^{*}$. Moreover, if $\mu $ is a mean on $X$ then
$$\int\!g(t)\, d\mu(t)\in \overline{co}\,\{g(t):t\in S\}.$$
We can write $g_{\mu}$ by
$$\int\!g(t)\, d\mu(t).$$
 \end{lem}

Next, we will need some concepts in locally convex spaces.\\
Consider     a   family of   seminorms  $Q$  on a        locally
convex space $X$ which determines the topology of $X$ and  a    seminorm $q \in Q$. Let $Y$ be a    subset  of
 $X$, we put   $q^{*}_{Y}(f)=\sup\{|f(y)|: y \in Y, q(y)\leq 1\}$ and      $q^{*}(f)=\sup\{|f(x)|: x \in X, q(x)\leq 1\}$, for every linear functional $f$ on $X$. Observe that, for each $x \in X$ that $q(x) \neq 0$ and $f \in X^{*}$, then   $|\langle x , f \rangle | \leq q(x) q^{*}(f)$.
 We will make use of the following Theorems.
\begin{thm}\cite{soori}\label{hahn1}
  Suppose that  $Q$ is a   family of   seminorms on a   real     locally
convex space $X$ which determines the topology of $X$ and $q \in Q$ is a continuous  seminorm and $Y$ is a  vector subspace of $X$ such that $Y \cap \{x \in X: q(x)=0\}=\{0\}$.     Let $f$ be a real    linear functional on $Y$ such that   $q^{*}_{Y}(f)< \infty$. Then there exists a continuous linear functional $h$ on $X$ that extends $f$ such that $q^{*}_{Y}(f)=q^{*}(h)$.
\end{thm}

\begin{thm}\cite{soori}\label{hahn2}
   Suppose that  $Q$ is a   family of   seminorms on a   real     locally
convex space $X$ which determines the topology of $X$ and  $q \in Q$   a nonzero  continuous seminorm.    Let $x_{0}$ be a point in $X$.  Then there exists  a continuous linear functional on $X$ such that $q^{*}(f)=1$ and  $f(x_{0})=q(x_{0})$.
\end{thm}

 \section{Main results}

In the following theorem,  we   prove that   $T_{\mu}$ inherits some properties of    representation $ \sc$ in Banach   spaces.

\begin{thm}
Suppose that  $ C $ is  a nonempty closed, convex subset of a
reflexive Banach space $  E$, $ S $   a semigroup,    $ \sc=\{T_{s}:s\in S\} $   a   representation of $S$ as    self  mappings on $C$  such that weak closure of  $\lbrace T_{t}x : t \in S \rbrace $ is weakly compact for each $ x \in C $ and $X$ be a subspace of $B(S)$ such that $ 1 \in X $ and  the mapping $t \rightarrow \left<T  _{t}x, x^{*}\right>$ be an element of $ X $ for
each $x \in C$       and $ x^{*} \in E$, and $\mu$ be a mean on $X$.
If we write $T_{\mu}x $ instead of
$\int\!T_{t}x\, d\mu(t),\/$
 then the following hold.
 \begin{enumerate}
   \item [(a)] Let  the mapping $t \rightarrow \left<T _{t}^{n}x- T  _{t}^{n}y, x^{*}\right>$ be an element of $ X $ for
each $x , y \in C$, $n \in \mathbb{N}$       and $ x^{*} \in E$. Let  $\mu$ be a mean on $X$ and  $ \sc=\{T_{s}:s\in S\} $  be  a   representation of $S$ as uniformly  asymptotically nonexpansive self mappings on $C$,  then $T_{\mu}$ is an asymptotically nonexpansive self mapping on $C$,
   \item [(b)]$T_{\mu}x=x $  for each  $x \in Fix(\sc) $,
   \item [(c)]$T_{\mu}x \in \overline{co}\,\{T_{t}x:t\in S\}$ for each $x\in C$,
   \item[(d)] if $  X$ is $ r_{s} $-invariant  for
  each $ s \in S $ and  $ \mu $ is right invariant, then $T_{\mu}T_{t} =T_{\mu} $ for each $ t \in S $,
   \item [(e)] let   $a \in C$ be a uniformly asymptotically attractive point of $\sc$ and  the mapping $t \rightarrow \left<a- T  _{t}^{n}x, x^{*}\right>$ be an element of $ X $ for
each $x \in C$, $n \in \mathbb{N}$       and $ x^{*} \in E$.  Then  $a  $ is an asymptotically attractive point of $T_{\mu}$,
 \item [(f)]  let $ \sc=\{T_{s}:s\in S\} $  be  a  representation of $S$    as the affine self  mappings on $C$, then      $T_{\mu}$ is an affine self mapping on $C$,
 \item [(g)]
 let $P$ be a self mappings on $C$ that commutes with  $T_{s}\in \sc=\{T_{s}:s\in S\} $ for each $s \in S$. Let  the mapping $t \rightarrow \langle PT_{t}x, x^{*}\rangle$ be an element of $ X $ for
each $x \in C$       and $ x^{*} \in E$.  Then      $T_{\mu}$ commutes  with $P$,
     \item [(h)]  let $ \sc=\{T_{s}:s\in S\} $  be  a  representation of $S$  as quasi nonexpansive  self mappings on $C$, then      $T_{\mu}$ is a $Fix(\sc)$-quasi nonexpansive self mapping on $C$,

          \item [(i)]  let $ \sc=\{T_{s}:s\in S\} $  be  a   representation of $S$ as $F$-quasi nonexpansive  self mappings on $C$ (for a subset $F \subseteq Fix(\sc)$),   then      $T_{\mu}$ is an $F$-quasi nonexpansive self mapping on $C$,
\item [(j)]  let $ \sc=\{T_{s}:s\in S\} $  be  a  representation of $S$ as strongly $F$-quasi nonexpansive  self mappings on $C$ (for a subset $F \subseteq Fix(\sc)$),     then      $T_{\mu}$ is an strongly $F$-quasi nonexpansive self mapping on $C$,
              \item [(k)]  let $ \sc=\{T_{s}:s\in S\} $  be  a  representation of $S$ as retraction  self mappings on $C$,  then      $T_{\mu}$ is a retraction   self mapping on $C$,
 \item [(l)]  let  $E=H$ be a Hilbert space   and  $ \sc=\{T_{s}:s\in S\} $  be  a   representation of $S$  as monotone  self mappings on $H$,  then      $T_{\mu}$ is a monotone   self mapping on $H$.

 \end{enumerate}
\end{thm}
\begin{proof} (a)  Since  $\sc$ is a   representation as uniformly  asymptotically nonexpansive  self mappings on $C$,  hence, from \eqref{dhhghj} and  from part  (b) of  Theorem 3. 1. 7 in \cite{rudin},  there exists an
integer $m_{0} \in \mathbb{N}$ such that $$\displaystyle \sup_{t} \|T_{t}^{n}x-T_{t}^{n}y\|\leq \|x-y\|$$ for all $n \geq m_{0}$, $x,y \in C$.
 Suppose that  $x_{1}^{*} \in J (T_{\mu}^{n}x - T_{\mu}^{n}y)$ and  $x, y \in C$.  We know from \cite{hi} that, for any $ \mu \in X^{\ast}$, there
exists a unique element $ f_{\mu}$ in $E$ such that
\begin{equation}\label{ygn}
  \left \langle f_{\mu} ,x^{\ast}\right\rangle=\mu_{s}\left\langle f(s),x^{\ast}\right \rangle
\end{equation}
for all $ x^{\ast}\in E^{\ast} $.    Then from \eqref{ygn} we have
\begin{align*}
\|T_{\mu}^{n}x - T_{\mu}^{n}y\|^{2} =& \langle T_{\mu}^{n}x - T_{\mu}^{n}y , x
_{1}^{*}  \rangle= \mu_{t} \langle T_{t}^{n}x - T_{t}^{n}y , x
_{1}^{*}  \rangle
 \\ \leq & \sup_{t}
\|T_{t}^{n}x - T_{t}^{n}y\|\|T_{\mu}^{n}x - T_{\mu}^{n}y\| \\ \leq &  \|x-y\|\|T_{\mu}^{n}x - T_{\mu}^{n}y\|,
\end{align*}
hence for all $n \geq m_{0}$,  $x,y \in C$  we have
\begin{align*}
\|T_{\mu}^{n}x - T_{\mu}^{n}y\| \leq & \|x-y\|,
\end{align*}
therefore we have
\begin{align*}
\limsup_{n\rightarrow \infty}\|T_{\mu}^{n}x - T_{\mu}^{n}y\| \leq   \|x-y\|.
\end{align*}

(b)
suppose that   $x \in Fix(\sc) $ and  $x^{*} \in  E^{*}$. Therefore  we have
\begin{equation*}
   \langle T_{\mu}x, x^{*}  \rangle= \mu_{t} \langle T_{t}x , x^{*}  \rangle= \mu_{t} \langle x , x^{*}  \rangle=  \langle x , x^{*}  \rangle
\end{equation*}

 (c) this assertion  concludes from  Lemma \ref{chimig}.

   (d) to prove   this assertion,  we have
\begin{equation*}
   \langle T_{\mu}(T_{s}x) , x^{*}  \rangle= \mu_{t} \langle T_{ts}x , x^{*}  \rangle= \mu_{t} \langle T_{t}x , x^{*}  \rangle=  \langle T_{\mu}x, x^{*}  \rangle,
\end{equation*}

 (e) Since  $a$ is a     uniformly  attractive  point,  hence, from \eqref{dhhghj2} and  from part  (b) of  Theorem 3. 1. 7 in \cite{rudin}, for each $x \in C$ there exists an
integer $m_{0} \in \mathbb{N}$ such that $$\displaystyle \sup_{t} \|a-T_{t}^{n}x\|\leq \|a-x\|$$ for all $n \geq m_{0}$.
  Suppose that  $x_{2}^{*} \in J (a-T_{\mu}^{n}x )$,  therefore from \eqref{ygn} we have,
\begin{align*}
\|a - T_{\mu}^{n}x\|^{2} =& \langle a - T_{\mu}^{n}x , x_{2}^{*}  \rangle= \mu_{t} \langle a - T_{t}^{n}x , x_{2}^{*}  \rangle
 \\ \leq & \sup_{t}
\|a - T_{t}^{n}x\|\|a - T_{\mu}^{n}x\| \\ \leq &  \|a-x\|\|a - T_{\mu}^{n}x\|,
\end{align*}
hence for all $n \geq m_{0}$   we have
\begin{align*}
\|a - T_{\mu}^{n}x\| \leq & \|a-x\|,
\end{align*}
therefore for each $x \in C$  we have
\begin{align*}
\limsup_{n\rightarrow \infty}\|a - T_{\mu}^{n}x\| \leq   \|a-x\|.
\end{align*}

(f)  Suppose that  $x_{1}^{*} \in  E^{*}$.   Then for all positive  integers  $\alpha , \beta$ such that $\alpha +\beta=1$,  $x,y \in C$  and $t \in S$  we have
\begin{align*}
 \langle T_{\mu}(\alpha x +\beta y) , x
_{1}^{*}  \rangle =& \mu_{t} \langle T_{t}(\alpha x +\beta y) , x
_{1}^{*}  \rangle
 \\ = & \mu_{t} \langle \alpha T_{t} x +\beta T_{t} y , x_{1}^{*}  \rangle
  \\ = & \alpha \mu_{t} \langle  T_{t} x   , x_{1}^{*}  \rangle +  \beta \mu_{t} \langle   T_{t} y , x_{1}^{*}  \rangle
    \\ = & \alpha  \langle  T_{\mu} x   , x_{1}^{*}  \rangle +  \beta  \langle  T_{\mu} y , x_{1}^{*}  \rangle
     \\ = &   \langle \alpha T_{\mu} x +\beta T_{\mu} y  , x_{1}^{*}  \rangle
\end{align*}
hence,    we have
\begin{align*}
T_{\mu}(\alpha x +\beta y) =\alpha T_{\mu} x +\beta T_{\mu} y.
\end{align*}

 (g) Let   $x_{1}^{*} \in  E^{*}$. Since $P$   commutes with  $T_{s}\in \sc=\{T_{s}:s\in S\} $ for each $s \in S$, then  from \eqref{ygn}, for each  $x\in C$  and $t \in S$  we have
\begin{align*}
 \langle T_{\mu}Px , x
_{1}^{*}  \rangle =& \mu_{t} \langle T_{t}Px , x
_{1}^{*}  \rangle
  \\ = &   \mu_{t} \langle P T_{t} x   , x_{1}^{*}  \rangle
    \\ = &    \langle P T_{\mu} x   , x_{1}^{*}  \rangle,
\end{align*}
then      $ T_{\mu}P= P T_{\mu}$.

 (h)Since $X$ is a subspace of $B(S)$ and $ 1 \in X $ and  the mapping $t \rightarrow \left<T  _{t}x, x^{*}\right>$ is an element of $ X $ for
each $x \in C$       and $ x^{*} \in E$, hence,   the mapping $t \rightarrow \left<T  _{t}x-f, x^{*}\right>$ is an element of $ X $ for
each $x \in C$       and $ x^{*} \in E$.
     For each $t \in S$   we have $\| T_{t} x -f\| \leq \|  x -f\|$ for each $f \in \text{Fix}(T_{t})$ and $x \in C$. Suppose that $f \in \text{Fix}(\sc)$ and
 $x_{2}^{*}
 \in J (T_{\mu} x -f)$, then  from \eqref{ygn}, we have
\begin{align*}
 \| T_{\mu} x -f\|^{2} =& \langle T_{\mu} x -f , x_{2}^{*}  \rangle= \mu_{t} \langle T_{t} x -f , x_{2}^{*}  \rangle
 \\ \leq & \sup_{t}
\|T_{t} x -f\|\|T_{\mu} x -f\| \\ \leq &  \|x -f\| \|T_{\mu} x -f\|,
\end{align*}
then we have
\begin{align*}
 \| T_{\mu} x -f\|   \leq   \|x -f\|,
\end{align*}
 then      $T_{\mu}$ is a $\text{Fix}(\sc)$-quasi nonexpansive self mapping on $C$.

 (i)  Let $ \sc=\{T_{s}:s\in S\} $  be  a    representation of $S$ as $F$-quasi nonexpansive   self mappings on $C$  that  $F \subseteq  \text{Fix}(\sc)$,
 then   for each $t \in S$   we have $\| T_{t} x -f\| \leq \|  x -f\|$ for each $f \in F$ and $x \in C$. Suppose that $f \in F$,    $x \in C$ and
 $x_{2}^{*}
 \in J (T_{\mu} x -f)$, then,  as in the proof of (h), from \eqref{ygn}, we have
\begin{align*}
 \| T_{\mu} x -f\|^{2} =& \langle T_{\mu} x -f , x_{2}^{*}  \rangle= \mu_{t} \langle T_{t} x -f , x_{2}^{*}  \rangle
 \\ \leq & \sup_{t}
\|T_{t} x -f\|\|T_{\mu} x -f\| \\ \leq &  \|x -f\| \|T_{\mu} x -f\|,
\end{align*}
then we have
\begin{align*}
 \| T_{\mu} x -f\|   \leq   \|x -f\|,
\end{align*}
 then      $T_{\mu}$ is an $F$-quasi nonexpansive self mapping on $C$.

(j) Let $ \sc=\{T_{s}:s\in S\} $  be  a   representation of $S$ as strongly $F$-quasi nonexpansive   self mappings on $C$ such that $F \subseteq Fix(\sc)$,
 then   for each $t \in S$   we have $\| T_{t} x -f\| < \|  x -f\|$ for each $x \in C \backslash F $ and $f \in F$. Suppose that $f \in F$, $x \in C \backslash F $ and
 $x_{2}^{*}
 \in J (T_{\mu} x -f)$, then  from \eqref{ygn}, we have
\begin{align*}
 \| T_{\mu} x -f\|^{2} =& \langle T_{\mu} x -f , x_{2}^{*}  \rangle= \mu_{t} \langle T_{t} x -f , x_{2}^{*}  \rangle
 \\ \leq & \sup_{t}
\|T_{t} x -f\|\|T_{\mu} x -f\| \\ < &  \|x -f\| \|T_{\mu} x -f\|,
\end{align*}
then we have
\begin{align*}
 \| T_{\mu} x -f\|   <   \|x -f\|,
\end{align*}
 then      $T_{\mu}$ is a strongly $F$-quasi nonexpansive self mapping on $C$.

 (k) Since $ T_{t}^{2}= T_{t}$ and   the mapping $t \rightarrow \left<T  _{t}x, x^{*}\right>$ is an element of $ X $ for
each $x \in C$       and $ x^{*} \in E$, hence  the mapping $t \rightarrow \left<T ^{2} _{t}x, x^{*}\right>$ is an element of $ X $ for
each $x \in C$       and $ x^{*} \in E$. Suppose that     $x \in C$ and
 $x_{1}^{*}
 \in E^{*}$, then  from \eqref{ygn}, we have
\begin{align*}
 \langle T_{\mu}^{2}x , x
_{1}^{*}  \rangle =& \mu_{t} \langle T_{t}^{2}x , x
_{1}^{*}  \rangle
  \\ = &   \mu_{t} \langle  T_{t} x   , x_{1}^{*}  \rangle
    \\ = &    \langle   T_{\mu} x   , x_{1}^{*}  \rangle,
\end{align*}
then      $ T_{\mu}^{2}=  T_{\mu}$.

(l) Since   $ T_{s} $  is monotone  for every $s \in S$,    then  we have $\langle T_{s} x - T_{s} y , x-y \rangle \geq 0 $ for every $x , y  \in H$ and $s \in S$.
 As in the proof of Theorem 1.4.1 in \cite{tn}  we know that  $\mu$ is positive i.e.,     $\langle \mu , f \rangle \geq 0$ for each $ f \in X$ that $f\geq 0$  . Then for each      $x , y  \in H$,    from \eqref{ygn}    we have
\begin{align*}
 \langle T_{\mu} x - T_{\mu} y , x-y \rangle =& \mu_{t} \langle T_{t} x - T_{t} y , x-y  \rangle
   \geq 0,
\end{align*}
then        $T_{\mu}$ is a monotone   self mapping on $H$.

\end{proof}

Now we study some properties of $T_{\mu}$ in locally convex spaces.

We will need the following   Theorem.

 \begin{thm}\label{rooh}  Let $ S $ be a semigroup, $E$   be a real   dual  locally convex space  with real predual   locally convex space $D$    and  $U$   a convex neighbourhood  of $0$ in  $D$ and $p_{U}$ be the associated
Minkowski functional.  Let $f : S \rightarrow E$
be a     function  such that  $\langle x , f (t) \rangle \leq 1$  for each  $t \in S$ and $x \in U$. Let $X$ be a subspace of $B(S)$ such that   the mapping $t \rightarrow  \langle   x , f (t)\rangle$ be an element of $ X $, for
each     $ x \in D$.   Then, for any $\mu \in X^{*}$,  there exists a unique element $F_{\mu} \in E$ such that
 $\langle  x , F_{\mu}  \rangle = \mu_{t} \langle   x , f (t)\rangle$ , for all    $ x \in D$.
Furthermore, if $1 \in X$ and $\mu$ is
a mean on $X$, then $F_{\mu}$ is contained in $ {\overline{{\text{co}}\{f(t) : t \in S\}}}^{w^{*}}$.
 \end{thm}
 \begin{proof}
  We define $F_{\mu}$ by $\langle x , F_{\mu} \rangle = \mu_{t} \langle   x , f (t)\rangle$ for all  $ x \in D$. Obviously, $ F_{\mu}$ is
  linear in $x$. Moreover, from  Proposition 3.8  in \cite{sob},  we have
\begin{align}\label{contin}
| \langle x , F_{\mu}\rangle |  = & | \mu_{t} \langle    x , f (t)\rangle |
  \leq   \sup_{t } |\langle    x , f (t)\rangle |.\|\mu\|\leq P_{U}(x).\|\mu\|,
\end{align}
 for all  $ x  \in D$. Let $(x_{\alpha})$ be a net in $  D$ that converges to $ x_{0}$.  Then  by \eqref{contin} we have
 \begin{align*}
  | \langle x_{\alpha} , F_{\mu}\rangle - \langle x_{0} , F_{\mu}\rangle|  = | \langle x_{\alpha}- x_{0} , F_{\mu}\rangle | \leq P_{U}( x_{\alpha}- x_{0}).\|\mu\|,
 \end{align*}
taking limit, since from Theorem 3.7  in \cite{sob},  $P_{U}$  is continuous, we have $F_{\mu}$  is continuous on $D$, hence  $F_{\mu} \in E$.

Now, let   $1 \in X$ and $\mu$ be
a mean on $X$. Then, there exists a net $\{\mu_{\alpha}\}$
of finite means on $X$ such that $\{\mu_{\alpha}\}$ converges to $\mu$ with the weak$^{*}$  topology on
$X^{*}$. We may consider that
\begin{align*}
  \mu_{\alpha}=\sum _{i=1}^{n_{\alpha}}\lambda_{\alpha,i}\delta_{t_{\alpha,{i}}}.
\end{align*}
Therefore,
\begin{align*}
  \langle  x ,  F_{\mu_{\alpha}}  \rangle= (\mu_{\alpha})_{t}  \langle  x , f(t) \rangle= \langle x ,   \sum _{i=1}^{n_{\alpha}}\lambda_{\alpha,i}f({t_{\alpha,{i}}})   \rangle, ( \forall x \in D, \forall \alpha),
\end{align*}
then   we have
\begin{align*}
 F_{\mu_{\alpha}}= \sum _{i=1}^{n_{\alpha}}\lambda_{\alpha,i}f({t_{\alpha,{i}}})\in {\text{co}}\{f(t) : t \in S\}, ( \forall \alpha ),
\end{align*}
now since,
\begin{align*}
  \langle  x ,  F_{\mu_{\alpha}} \rangle= (\mu_{\alpha})_{t}  \langle   x , f(t) \rangle \rightarrow \mu_{t}  \langle    x , f(t) \rangle= \langle   x , f(t) \rangle, (x \in D),
\end{align*}
 $\{F_{\mu_{\alpha}}\}$ converges to $F_{\mu}$ in the weak$^{*}$  topology. Hence
\begin{align*}
 F_{\mu}\in \overline{{\text{co}}\{f(t) : t \in S\}}^{w^{*}},
\end{align*}
we can write  $F_{\mu}$ by  $\int f(t) d\mu(t)$.

 \end{proof}

In the following theorem,  we   prove that   $T_{\mu}$ inherits some properties of    representation $ \sc$ in locally convex    spaces.

\begin{thm}\label{bshhwj}
Let $ S $ be a semigroup,  $ C $      a closed convex     subset of a real
   locally convex space $E$. Let
$G = (V(G),E(G))$ a directed graph such that $V(G) =C$. Let $\mathcal{B}$ be a
base at 0 for the topology consisting of convex, balanced sets.  Let $Q=\{q_{V}: V \in \mathcal{B} \}$
which      $q_{V}$  is the associated
Minkowski functional with  $V$.   Let $ \sc=\{T_{s}:s\in S\} $  be   a representation of $S$ as     $Q$-$G$-nonexpansive    mappings from $C$ into
itself  and $X$ be a subspace of $B(S)$ such that $ 1 \in X $ and $\mu$ be a mean on $X$  such that the mapping $t \rightarrow \left<T_{t}x, x^{*}\right>$ is an element of $ X $ for
each $x \in C$ and $ x^{*} \in E^{*}$.
If we write $T_{\mu}x $ instead of
$\int\!T_{t}x\, d\mu(t),\/$
 then the following hold.
 \begin{enumerate}
   \item [(i)]  $T_{\mu}$ is a $Q$-$G$-nonexpansive mapping from $C$ into $C$,
   \item [(ii)]  $T_{\mu}x=x $  for each  $x \in Fix(\sc) $,
   \item [(iii)] If moreover $E$   is a real   dual  locally convex space  with real predual   locally convex space $D$  and  $ C $      a $w^{*}$-closed convex     subset  of $E$   and  $U$   a convex neighbourhood  of $0$ in  $D$ and $p_{U}$ is the associated
Minkowski functional.      Let   the mapping $t \rightarrow \langle z , T _{t}x \rangle $ is an element of $ X $ for
each $x \in C$ and $z \in D$ then  $T_{\mu}x \in \overline{co \,\{T_{t}x:t\in S\}}^{w^{*}}$ for each $x\in C$,
   \item[(iv)]     if $  X$ is $ r_{s} $-invariant  for
  each $ s \in S $ and  $ \mu $ is right invariant, then $T_{\mu}T_{t} =T_{\mu} $ for each $ t \in S $,
   \item [(v)]    let $a \in E$ is an $Q$-$G$-attractive point of $\sc$ and    the mapping $t \rightarrow \left<a-T  _{t}x, x^{*}\right>$ be an element of $ X $ for
each $x \in C$       and $ x^{*} \in E$,  then  $a  $ is an $Q$-$G$-attractive point of $T_{\mu}$.
 \end{enumerate}
\end{thm}
\begin{proof}
 (i) Let $x, y \in C$ and   $V \in \mathcal{B}$.  By  Proposition 3.33  in \cite{sob}, the topology on $E$ induced by
$Q$ is the original topology on $E$.
By Theorem 3.7  in \cite{sob}, $q_{V}$ is a continuous seminorm and from  Theorem 1.36 in \cite{rudin}, $q_{V}$ is a nonzero seminorm because if $x \notin V$  then $q_{V}(x) \geq 1$, hence   from Theorem \ref{hahn2},  there exists a functional $x^{*}_{V} \in X^{*}$ such that    $q_{V}( T_{\mu}x - T_{\mu}y) = \langle  T_{\mu}x - T_{\mu}y , x^{*}_{V}  \rangle$ and $q_{V}^{*}(x^{*}_{V})=1$, and since    from Theorem 3.7 in  \cite{sob},  $q_{V} (z)\leq 1$ for each $z \in V$, we conclude that  $\langle z ,  x^{*}_{V}\rangle \leq 1  $ for all $z \in V$.  Therefore     from Theorem  3.8 in \cite{sob}, $\langle z , x^{*}_{V}   \rangle \leq  q_{V}(z)$ for all $z \in E$.  Hence   for
each $t \in S$,    $x, y \in C$ that  $(x, y) \in E(G)$ and $ x^{*} \in E^{*}$,    from \eqref{ygn},  we have
\begin{align*}
q_{V}( T_{\mu}x - T_{\mu}y) =& \langle  T_{\mu}x - T_{\mu}y , x^{*}_{V}  \rangle= \mu_{t} \langle  T_{t}x - T_{t}y , x^{*}_{V}  \rangle
 \\ \leq &\| \mu\| \sup_{t}| \langle  T_{t}x - T_{t}y , x^{*}_{V}  \rangle | \\ \leq & \sup_{t}
q_{V}( T_{t}x - T_{t}y) \\ \leq &  q_{V}(x-y),
\end{align*}
 then  we have
\begin{align*}
q_{V}( T_{\mu}x -  T_{\mu}y) \leq & q_{V}(x-y),
\end{align*}
  for all   $V \in \mathcal{B}$.

(ii)
Let  $x \in Fix(\sc) $ and  $x^{*} \in  E^{*}$. Then we have
\begin{equation*}
   \langle T_{\mu}x, x^{*}  \rangle= \mu_{t} \langle T_{t}x , x^{*}  \rangle= \mu_{t} \langle x , x^{*}  \rangle=  \langle x , x^{*}  \rangle
\end{equation*}

 (iii) this assertion  concludes from  Theorem \ref{rooh}.

   (iv) for  this assertion, note that
\begin{equation*}
   \langle T_{\mu}(T_{s}x) , x^{*}  \rangle= \mu_{t} \langle T_{ts}x , x^{*}  \rangle= \mu_{t} \langle T_{t}x , x^{*}  \rangle=  \langle T_{\mu}x, x^{*}  \rangle,
\end{equation*}
 (v) Let $x \in C$ and   $V \in \mathcal{B}$.
 From Theorem \ref{hahn2},  there exists a functional $x^{*}_{V} \in X^{*}$ such that    $q_{V}(a - T_{\mu}x) = \langle a - T_{\mu}x , x^{*}_{V}  \rangle$ and $q_{V}^{*}(x^{*}_{V})=1$. Since    from Theorem 3.7 in  \cite{sob},  $q_{V} (z)\leq 1$ for each $z \in V$, we conclude that  $\langle z ,  x^{*}_{V}\rangle \leq 1  $ for all $z \in V$.  Therefore     from Theorem  3.8 in \cite{sob}, $\langle z , x^{*}_{V}   \rangle \leq  q_{V}(z)$ for all $z \in E$.  Hence   for
each $t \in S$,  $x, y \in C$ that  $(x, y) \in E(G)$ and $ x^{*} \in E^{*}$,    from \eqref{ygn},  we have
\begin{align*}
q_{V}(a - T_{\mu}x) =& \langle a - T_{\mu}x , x^{*}_{V}  \rangle= \mu_{t} \langle a - T_{t}x , x^{*}_{V}  \rangle
 \\ \leq &\| \mu\| \sup_{t}| \langle   a - T_{t}x , x^{*}_{V}  \rangle | \\ \leq & \sup_{t}
q_{V}(  a - T_{t}x) \\ \leq &  q_{V}(a-x),
\end{align*}
 then  we have
\begin{align*}
q_{V}( a - T_{\mu}x) \leq & q_{V}(a-x),
\end{align*}
  for all   $V \in \mathcal{B}$.
  \end{proof}
  \section{Discussion}
In this paper, we investigate  some  properties of the mapping $T_{\mu}$  introduced by a representation. $T_{\mu}$ is an important mapping in the literature, for example see \cite{saeidi, eees, soori}.

\section{Conclusion}
In this paper, we prove that some properties of the mapping in  the representation $ \sc=\{T_{s}:s\in S\} $ can be transferred to  the mapping  $T_{\mu}$ introduced by a mean on a subspace  of $B(S)$,   for example nonexpansiveness, quasi-nonexpansiveness, strongly quasi-nonexpansiveness, monotonicity, retraction property and another properties    in Banach spaces, and  $Q$-$G$-nonexpansiveness  using a directed graph in locally
convex spaces.

\section{ Abbreviations}
Not applicable
\section{Declarations }
\subsection{Availability of data and material}
 Please contact author for data requests.
\subsection{Funding}
Not applicable
\section{  Acknowledgements}
 The first author is  grateful to  the University of Lorestan for their support.
 \section{Competing interests}
 The authors declare that they have no competing interests.
 \section{Authors' contributions}
 All authors contributed equally to the manuscript, read and approved the final manuscript.

\bibliographystyle{amsplain}

\end{document}